\documentclass{amsart}
\usepackage[T1]{fontenc}
\usepackage{amsmath,amsthm,amscd,amssymb}
\usepackage{graphicx}

\newtheorem{theorem}{Theorem}[section]
\newtheorem{proposition}[theorem]{Proposition}

\newtheorem{lemma}[theorem]{Lemma}
\title[Automated computation of robust normal forms]{Automated computation of \\robust normal forms \\of planar analytic vector fields}
\date{October 29, 2008}
\author{Tomas Johnson, Warwick Tucker}
\email{johnson@math.uu.se, warwick.tucker@math.uib.no}
\address{Department of Mathematics, Uppsala University, Box 480, 751 06 Uppsala, Sweden}
\address{Department of Mathematics, University of Bergen, Johannes Brunsgate 12, 5008 Bergen, Norway}

\begin{document}

\maketitle

\begin{abstract}
We construct an auto-validated algorithm that calculates a close to identity change of variables which brings a general saddle point into a normal form. The transformation is robust in the underlying vector field, and is analytic on a computable neighbourhood of the saddle point. The normal form is suitable for computations aimed at enclosing the flow close to the saddle, and the time it takes a trajectory to pass it. Several examples illustrate the usefulness of this method.
\end{abstract}

\section{Introduction}
It is well-known that computing a trajectory in the close vicinity of a fixed point is associated with many problems. Numerical integration schemes (silently) break down when the vector field tends to zero, and this usually results in completely inaccurate results. Indeed, as the norm of the vector field decreases, the flow-time needed to pass a saddle increases without bound. This means that no integration scheme, rigorous or not, will function properly is this situation. There are, however, many instances where it is necessary to be able to follow the flow of a vector field arbitrarily close to a saddle. 

We present a completely automated, rigorous method that produces analytical estimates on the flow close to a given saddle. Equally important, it produces explicit bounds, for a given accuracy of the analytic estimates, on the size of the neighbourhood of the saddle on which the information is valid. This avoids the need to numerically integrate the flow near a saddle: once a trajectory comes close to the saddle, the bounds produced by our method give enclosures of where the trajectory exits the neighbourhood, and its associated flow-time.

The approach is based on constructing a carefully chosen change of variables, that bring the original vector field into the robust normal form presented in \cite{T02,T04}. The present paper can be seen as a quantitative companion to \cite{T04}, where several qualitative properties of robust normal forms are proved. Many of the ideas behind the algorithm can be found in \cite{T02}, where they were used for establishing that the Lorenz equations support a strange attractor. In the present study we develop an algorithm for general planar real analytic vector fields.

Consider the planar vector field
\begin{equation}\label{DEq}
\dot x = \Lambda x +F(x),
\end{equation}
with $\Lambda\in \mathcal S$, where $\mathcal S:=\{\rm diag(\lambda_s, \lambda_u ) \, : \lambda_s <0, \lambda_u >0\},$ and where $F$ is an analytic function, with $F(x)=O(x^2)$. Note that any vector field with a saddle fixed point can (locally) be brought into this form by an affine change of variables.

The purpose of this paper is to describe, and implement, an algorithm that finds a square centred at the saddle in which we can enclose a trajectory and its flow-time passing near the saddle. The output of the program includes estimates on the norms of the change of variables, its inverse, and the nonlinear part of the normal form, as well as the flow-time for passing the saddle.

\section{Theoretical background and notation}
This paper addresses the algorithmic aspects of the planar case of the robust normal forms introduced in \cite{T02}, and formalised in \cite{T04}. In order to simplify the formulae, we use vector and multiindex notation. The components of a vector are indexed by $s$ and $u$ for the stable and unstable direction, respectively. To make the presentation self-contained, we revise the necessary concepts from \cite{T04}, but refer the reader to that paper for proofs and additional details.

The structure of (\ref{DEq}) implies that the stable and unstable manifold of the origin are tangent to the coordinate axes. Rather than attempting to find a coordinate change that completely linearises (\ref{DEq}) in accordance with Siegel's theorem \cite{S52}, we compute normal forms that are robust in the sense that the set of eigenvalues where they exist is open and dense. This is crucial from a computational point of view, as we often only have an approximate knowledge of the eigenvalues. Our aim is to change (\ref{DEq}) into the normal form 
\begin{equation}\label{DEq2}
\dot y = \Lambda y +G(y),
\end{equation}
by an analytic change of coordinates, $x=y+\phi(y)$. We require that $G$, the non-linear part of the new vector field, is such that the invariant manifolds of the saddle are not only tangent to the coordinate axes, but actually coincide with them locally. We also require that the vector field is at least linear on these invariant manifolds. That is, if we let $d(y)=\min(|y_s|,|y_u|)$, then we ensure that $G_i(y)=O(d(y)^l)$, where $l$ is the {\it order of flatness}. This means that if $g_m$ is a non-zero coefficient in the series expansion of $G$, then $m_s \geq l$ and $m_u\geq l$. We call the non-negative number $|m| = m_u + m_s$ the \textit{order} of $m$, and define the set $\tilde{\mathbb N}^2 = \{m\in \mathbb N^2\colon |m|\ge 2\}$.

To formalise, let us split the space of multi-exponents into the sets
\begin{eqnarray*}
\mathbb{V}_l & := & \{m\in \tilde{\mathbb N}^2 \, : m_s<l \textrm{ or } m_u<l\},\\
\mathbb{U}_l & := & \{m\in \tilde{\mathbb N}^2 \, : m_s\geq l \textrm{ and } m_u \geq l\}.
\end{eqnarray*}
Now we can define the set of admissible linear parts of (\ref{DEq}) that we consider:
$$
\mathcal F_l := \{\Lambda \in \mathcal S \, : m\in \mathbb V_l \Rightarrow m\lambda-\lambda_i\neq 0, i=u,s\}.
$$
It is proved in \cite{T04} that $\mathcal F_l$ is open and has full Lebesgue measure in $\mathcal S$. We will often use the notion of filters of a (formal) power-series: if $f(x)=\sum_{|m|\geq 2} \alpha_m x^m$, we use the notation 
$$
[f]_{U_l}=\sum_{m\in U_l} \alpha_m x^m,\quad [f]_{V_l}=\sum_{m\in V_l} \alpha_m x^m, \text{ and } [f]_{m}=\alpha_m.
$$
Also, we let $f^d$ denote the partial sum of the first $d$ terms of $f$. We use the norms $|y|=\max{(|y_s|,|y_u|)}$ and $||f||_r = \max\{|f(y)| \, : |y|<r\}$. The $r$-disc is denoted by $\mathfrak B_r$, and at times we use the notation $\check \lambda$ and $\hat \lambda$, to denote the eigenvalue with the smallest and largest absolute value, respectively.

We are now ready to state the two main theorems from \cite{T04}:

\begin{theorem}\label{thm1}
Given an integer $l\geq 2$ and a system $\dot x = \Lambda x + F(x)$ where $F(x) = \sum_{|m|\geq2}a_mx^m$ is analytic, and $\Lambda\in\mathcal F_l$, there exists positive constants $r_0, r_1, K_0, K_1$ and an analytic, close to identity change of variables $x = y+\phi(y)$ with $$||\phi||_r\leq K_0r^2 \quad (r<r_0),$$
such that $\dot x = \Lambda x+ F(x)$ is transformed into the normal form $\dot y = \Lambda y + G(y)$ satisfying $[G(y)]_{U_l} = G(y)$ and $$||G||_r\leq K_1 r^{2l} \quad (r<r_1).$$
\end{theorem}
In the second theorem, we let $\Psi$ denote the flow of $\dot y = \Lambda y + G(x)$.
\begin{theorem}\label{thm2}
Under the same conditions as in Theorem \ref{thm1}, and given any $\kappa>0$ sufficiently small, there exists $r>0$ such that for any trajectory in $\mathfrak B_r$ starting from $|x_s|=r$, we have the following enclosure of its point of exit:
$$\Psi_u(y,\tau_e(y)) = {\rm sign}(y_u)r;$$
$$ r\left(\frac{|y_u|}{r}\right)^{\frac{|\lambda_s|+\kappa}{\lambda_u-\kappa}}\leq \Psi_s(y,\tau_e(y)) \leq r\left(\frac{|y_u|}{r}\right)^{\frac{|\lambda_s|-\kappa}{\lambda_u+\kappa}},$$
where $\tau_e(y)$ (the exit time) denotes the time spent inside of $\mathfrak B_r$:
$$ \frac{1}{\lambda_u+\kappa} \log \frac{r}{|y_u|}\leq \tau_e(y) \leq \frac{1}{\lambda_u-\kappa}\log \frac{r}{|y_u|}$$
\end{theorem}

We will also use the following lemma from \cite{T04}.

\begin{lemma}\label{gEst}
If $(\lambda_s, \lambda_u)$ are non-resonant for $m\in \mathbb{V}_l$, then the divisors $m\lambda-\lambda_i$ are bounded away from zero. Furthermore, for all orders
$|m|\geq l+\left\lceil(l-1)\left|\frac{\hat \lambda}{\check \lambda}\right|\right\rceil$, we have the following sharp lower bound:
$$|m\lambda-\lambda_i|\geq |(|m|-l)\check \lambda+(l-1)\hat \lambda|$$
\end{lemma}

Finally, the following lemma, which in principle appears in \cite{T02}, will be used.

\begin{lemma}\label{phiInvN}
If $r<r_0(1-K_0r_0)$, then $\phi$ has a well-defined inverse, $y=x+\phi^{-1}(x)$ in $|x|<r^*=r-||\phi||_r$, satisfying $$||\phi^{-1}||_{r^*}\leq||\phi||_r$$
\end{lemma}

To prove the convergence of $\phi$ and $G$ we procede as in e.g. \cite{H76,SM71}, and use the method of majorants. If $f,g :\mathbb {C}^n\rightarrow \mathbb {C}^n$, are two formal power series and $|f_m|<g_m$, for all multiindices $m$, and all the coefficients of $g$ are real and positive, we say that $g$ {\it majorises} $f$, denoted by $f\prec g$. Thus, the convergence radius of $f$ is at least as large as $g$'s. We will majorise in two steps; given some $f:\mathbb {C}^2\rightarrow \mathbb {C}^2$, we construct $g:\mathbb {C}^2\rightarrow \mathbb {C}$ such that $f_i\prec g$, for all $i$, and then construct $h:\mathbb {C}\rightarrow \mathbb {C}$ such that $g(z,z)\prec h(z)$.

\section{The algorithm}
In this section we will describe an algorithm that computes explicit bounds on the constants $r_0$, $r_1$, $K_0$, and $\kappa$ appearing in Theorems \ref{thm1} and \ref{thm2}. This allows us to interrupt a numerical integration scheme of the original vector field (\ref{DEq}), and instead use the analytical bounds from Theorem \ref{thm2} to enclose the flow on passing the saddle, together with bounds on the time it takes to pass the saddle. 

The main ideas of the algorithm appear in \cite{T02}, where robust normal forms are computed for the Lorenz system. In \cite{T02}, however, the algorithm was designed exclusively for that particular system. The purpose of this paper is to construct a general algorithm, that will take any planar vector field of the form (\ref{DEq}), and transform it into (\ref{DEq2}), together with explicit bounds on the aforementioned constants.

We note that several heuristic constants appear in the algorithm: $\iota$, $\eta$, $\mu$, $\rho$, $N_G$, $\epsilon_\phi$, $\epsilon_G$, and $\epsilon$. In the actual implementation all of these can be set by the user in a configuration file. The algorithm has been implemented in a {\tt C++} program using the C-XSC package \cite{CXSC,HH95} for interval arithmetic \cite{AH83,Mo66,Mo79,Ne90}. For automatic differentiation \cite{G00} we use a modified version of the Taylor arithmetic package \cite{BHK05}.

\subsection{Outline}
The algorithm has four main parts that will be described in detail below: 
\begin{itemize}
\item [1.] we compute an $l$ such that $\Lambda\in \mathcal F_l$, with $l\leq \iota$, where $\iota$ is a user-provided order of flatness. 
\item [2.] we compute the first few terms of the formal power series solution of the functional equation for the change of coordinates $x=y+\phi(y)$ using automatic differentiation. These first terms are used to estimate bounds on a majorant of $\phi$. These estimates are then, finally, used to prove the analyticity of $\phi$, using induction.
\item [3.] we do the same kind of estimates for $G$; we compute some terms in the formal power series solution of a second functional equation, and use these to prove the analyticity of $G$ by induction. 
\item [4.] using the estimates on the coefficients of the analytic functions $\phi$, and $G$, we estimate the constants $K_0$ and $\kappa$ that enable the user to switch from a numerical integration scheme to the analytic estimates from Theorem \ref{thm2}.
\end{itemize}

\subsection{Verifying $\Lambda\in \mathcal F_l$}
We want to determine $l$ such that $\Lambda \in \mathcal F_l$. We will do this by first constructing $\mathbb V_l$, and then removing its members that cause resonances.
\begin{proposition}\label{LambdaProp}
If, for $i=1,...,l$ 
$$i\frac{-\lambda_s}{\lambda_u} \notin \mathbb N 
\quad\textrm{ and }\quad 
i\frac{\lambda_u}{-\lambda_s} \notin \mathbb N, $$
then  $\Lambda \in \mathcal F_{l+1}$.
\end{proposition}
\begin{proof}

For $l = 1$, we note that $\mathbb V_1 = \{(0,i), (i,0)\}_{i\geq 2}$. The potential resonances are given by
$$
m_u \lambda_u - \lambda_i = 0
\quad\text{and}\quad
m_s \lambda_s - \lambda_i = 0
$$
and it is clear that no member of $\mathbb V_1$ satisfies any of these two equations. Hence, $\mathcal F_1 = \mathcal S$.

For $l\geq 1$, we have the recursive relation:
$$
\mathbb V_{l+1} = \mathbb V_l \cup \{(l,i), (i,l)\}_{i\geq{l}}.
$$ 
Thus we only have to consider the following potential resonances:
$$
m_u \lambda_u + l\lambda_s - \lambda_i = 0
\quad\text{and}\quad
l\lambda_u + m_s \lambda_s - \lambda_i = 0
$$
with $m_u, m_s \ge l$. For the case $i = s$, we get
$$
m_u \lambda_u + (l-1)\lambda_s = 0
\quad\text{and}\quad
l\lambda_u + (m_s-1) \lambda_s = 0
$$
with solutions $m_u = (l-1)\tfrac{-\lambda_s}{\lambda_u}$ and $m_s = 1 + l\tfrac{\lambda_u}{-\lambda_s}$, respectively. Analogously, for the case $i = u$, we get
$$
(m_u-1) \lambda_u + l\lambda_s = 0
\quad\text{and}\quad
(l-1)\lambda_u + m_s \lambda_s = 0
$$
with solutions $m_u = 1 + l\tfrac{-\lambda_s}{\lambda_u}$ and $m_s = (l-1)\tfrac{\lambda_u}{-\lambda_s}$, respectively. Therefore it suffices to enforce
\begin{equation}\label{inFl}
i\frac{-\lambda_s}{\lambda_u} \notin \mathbb N 
\quad\textrm{ and }\quad 
i\frac{\lambda_u}{-\lambda_s} \notin \mathbb N, 
\qquad i=1,...,l.
\end{equation}
to establish that $\Lambda \in \mathcal F_{l+1}$.
\end{proof}
It follows from Proposition \ref{LambdaProp} that we have the relation:
$$
\mathcal F_{l+1}=\mathcal F_{l} \backslash\{\Lambda\in \mathcal S \, : l\frac{-\lambda_s}{\lambda_u} \in \mathbb N, 
\textrm{ or } l\frac{\lambda_u}{-\lambda_s} \in \mathbb N\}.
$$
To write a program that checks 
the condition in Proposition \ref{LambdaProp}
is simple, and the algorithm returns a lower estimate on the largest $l$ less than $\iota$, such that $\Lambda\in \mathcal F_l$.

\subsection{Computing $\phi$ and its radius of convergence}
By inserting $x=y+\phi(y)$ into (\ref{DEq}), differentiating directly, and comparing the sides, we get:
$$(I + D\phi )\dot y=\Lambda(y+\phi(y))+F(y+\phi(y)).$$
By inserting into (\ref{DEq2}), and simplifying, we get:
\begin{equation}\label{rec}
D\phi(y)\Lambda y-\Lambda\phi (y)=F(y+\phi(y))-D\phi(y)G(y)-G(y).
\end{equation}
Let $L_{\Lambda}$ be the operator $$L_{\Lambda}\phi=D\phi(y)\Lambda y-\Lambda\phi(y),$$
where we note that $(L_{\Lambda}(y_i^m))_i=(m\lambda-\lambda_i)y_i^m$.

Recall, we want to compute a normal form (\ref{DEq2}) which is $l$-flat, that is $[G]_{\mathbb{U}_l}=G$, and the non-flat terms in (\ref{rec}), which we want to cancel with $\phi$, come from $F$. Therefore, by filtering on the component level, we get the following two functional equations for $\phi_i$ and $G_i$:
\begin{equation}\label{phiRec}
(L_{\Lambda}\phi)_i=[F_i(y+\phi(y))]_{\mathbb{V}_l} 
\end{equation}

\begin{equation}\label{GRec}
G_i=[F_i(y+\phi(y))]_{\mathbb{U}_l} -\frac{\partial \phi_i}{\partial y_s}G_s(y)-\frac{\partial \phi_i}{\partial y_u}G_u(y).
\end{equation}
Since $\Lambda \in \mathcal F_l$, and $[\phi]_{\mathbb{V}_l}=\phi$ by construction, we can solve (\ref{phiRec}) recursively, 

To bound the solutions of (\ref{phiRec}) we want to procede as in \cite{T02}, and prove the convergence of the change of variables using majorants and induction. Two heuristic constants $n_0$, and $n_1>n_0$ are needed. They determine the range of coefficients of the formal power series of $\phi$, that should be used in the induction proof.
Let $$N(l):= l + \left\lceil(l-1)\left|\frac{\hat \lambda}{\check \lambda}\right| \right\rceil,$$ be the constant from Lemma \ref{gEst} from which the explicit lower bound holds. For the induction to work it is required that $n_1>N(l)$.

We put $n_1=\lceil(1+\eta)N(l)\rceil$, and $n_0=\lfloor\frac{1+\mu}{2}N(l)\rfloor$, where $\eta>0$, and $-1<\mu<\eta$ are two given constants.

Let $\phi_i(x)=\sum_{|m|=2}^\infty \alpha_{i,m}x^m$ be the sought change of variables. We will compute the $\alpha_{i,m}$'s with $|m|\leq n_1$ using automatic differentiation, and then put $\hat \alpha_{k}=\sum_{|m|=k} \max (|\alpha_{s,m}|, |\alpha_{u,m}|)$. The $\hat \alpha_k$'s will be used as the first terms in a majorant of $\phi_s$ and $\phi_u$. Sometimes we will use $\hat \alpha_1=1$, to simplify the argument of some functions.

To calculate $\alpha_{i,m}$, with $|m|=k$, we evaluate a $k$-Taylor model of $F_i\left(x+\phi^{k-1}(x)\right)$, and divide its $m$th term by $m\lambda-\lambda_i$:
\begin{equation}\label{compAlpha}
\alpha_{i,m} = \frac{\left[F_i\left(x+\phi^{k-1}(x)\right)\right]_m}{|\lambda m - \lambda_i|}
\end{equation}

Note, the coefficients at a certain level only depend on the previous levels. This is because $F$ does not contain constant or linear terms.

If $n_0$ and $n_1$ are sufficiently large, then the first terms computed above are a good approximation of a majorant $\hat \phi$, and we use these to determine an approximate radius of convergence for $\hat \phi$. The validity of this radius of convergence will be proved later.
Therefore we determine, using a least squares estimator, constants $C$ and $M$, such that $$\hat \alpha_k\leq CM^k,\quad n_0<k\leq n_1.$$ 
Thus, a candidate radius of convergence is $s:=\frac{1}{M}$, which needs to be verified. 

We will consider a slightly larger majorant of $\phi_i$. 
If $$F_i(x)=\sum_{|m|=2}^\infty c_{i,m}x^m,$$ 
we define $$\hat c_{k}:=\sum_{|m|=k} \max (|c_{s,m}|, |c_{u,m}|),$$ and set 
$$\hat F:=\sum_{k=2}^\infty \hat c_k x^k.$$ $\hat F$ is clearly a majorant of $F_i$. 
We define, 

\begin{equation}\label{Adef}
A:=\sum_{k=2}^\rho \hat c_k s^{k-2} + 
\left(\frac{||F_s||_{2s}+||F_u||_{2s}}{s^2}\left(\frac{1}{2}\right)^\rho(\rho+3)\right),
\end{equation}
where $\rho$ is a given natural number.

\begin{lemma}
$\hat F(x)\leq A|x|^2$, on $|x|<s$.
\end{lemma}
\begin{proof}
The terms of $\hat F$ up to order $\rho$ are clearly bounded by the left sum in (\ref{Adef}), since $\hat c_k\geq 0$, and $|x|<s$. For the coefficients $c_{i,m}$, standard Cauchy-estimates give $|c_{i,m}|\leq \frac{||F_i||_\zeta}{\zeta^m}$. Thus, since there are $(k+1)$ terms with $|m|=k$, 
$$\hat c_k \leq  (k+1)\frac{||F_s||_\zeta+||F_u||_\zeta}{\zeta^k} $$
Using $\zeta=2s$, this yields
\begin{equation}
\begin{array}{ccl}
\sum_{k=\rho+1}^\infty \hat c_k x^k & \leq & (||F_s||_{2s}+||F_u||_{2s})\sum_{k=\rho+1}^\infty(k+1) \left(\frac{|x|}{2s}\right)^k \\ \\
& \leq & \frac{||F_s||_{2s}+||F_u||_{2s}}{4s^2}x^2\sum_{k=\rho+1}^\infty(k+1) \left(\frac{|x|}{2s}\right)^{k-2} \\
& \leq & \frac{||F_s||_{2s}+||F_u||_{2s}}{4s^2}x^2\left(2\sum_{k=\rho+1}^\infty \left(\frac{1}{2}\right)^{k-2} \right. \\
& & + \left. \sum_{k=\rho+1}^\infty (k-1)\left(\frac{1}{2}\right)^{k-2}\right) \\
& = & \frac{||F_s||_{2s}+||F_u||_{2s}}{s^2}x^2 \left(\frac{1}{2}\right)^\rho\left(\rho+3\right)
\end{array}
\end{equation} 
\end{proof}
Let $$\Omega(k):=\left|(k-l)\check \lambda+(l-1)\hat \lambda\right|,$$ be the lower bound on $|m\lambda-\lambda_i|$ from Lemma \ref{gEst}.
\begin{proposition}\label{phiConvProp}
If 
$$\frac{A(s)}{\Omega(n_1+1)}\left(2\sum_{k=1}^{n_0}\hat \alpha_k s^{k}+n_1C\right) < 1,$$
then $\phi_i$ is analytic on $\mathfrak{B}_s$.
\end{proposition}
\begin{proof}
By the above lemma, $\phi_i$ is majorised by $\hat \phi$, where $\hat \phi^{n_1}=\sum_{k=2}^{n_1} \hat \alpha_k x^k$ is as above, and 
$$\hat \alpha_k = \frac{A}{\Omega(k)}\left[(\phi^{k-1}(r))^2\right]_k, \quad k>n_1.$$
Note that $n_0$ and $n_1$ are constructed so that $\max(2n_0,N_l)<n_1$. Thus, (for $n\geq n_1$) $$\alpha_{n+1}=\frac{A}{\Omega(n+1)}\sum_{k=1}^{n}\hat \alpha_k \hat \alpha_{n+1-k}.$$
Assume, for some $n\geq n_1$, we have proved that $\hat \alpha_k\leq CM^k$ holds for $n_0<k\leq n$. If we can prove that $\hat \alpha_{n+1}\leq CM^{n+1}$, the convergence of $\phi$ follows by induction.

$$\begin{array}{ccl}  
\alpha_{n+1} & = & \frac{A}{\Omega(n+1)}\left(\sum_{k=1}^{n_0}\hat \alpha_k \hat \alpha_{n+1-k}
+\sum_{k=n_0+1}^{n-n_0}\hat \alpha_k \hat \alpha_{n+1-k} \right.\\ 
& & + \left.\sum_{k=n-n_0+1}^{n}\hat \alpha_k \hat \alpha_{n+1-k}\right) \\
& = & \frac{A}{\Omega(n+1)}\left(2\sum_{k=1}^{n_0}\hat \alpha_k \hat \alpha_{n+1-k}
+\sum_{k=n_0+1}^{n-n_0}\hat \alpha_k \hat \alpha_{n+1-k}\right) \\ \\
& & \left(\textrm{use the induction hypothesis, }\hat \alpha_k\leq CM^k, \quad n_0<k \leq n\right)\\ \\
& \leq & \frac{A}{\Omega(n+1)}\left(2\sum_{k=1}^{n_0}\hat \alpha_k CM^{n+1-k}
+\sum_{k=n_0+1}^{n-n_0} C^2M^{n+1}\right) \\
& = &\frac{A}{\Omega(n+1)}\left(2\sum_{k=1}^{n_0}\hat \alpha_k M^{-k}+(n-2n_0)C\right)CM^{n+1} \\
& \leq & \frac{A}{\Omega(n+1)}\left(2\sum_{k=1}^{n_0}\hat \alpha_k M^{-k}+nC\right)CM^{n+1}
\end{array}$$
The expression before $CM^{n+1}$ is decreasing in $n$ for $n\geq n_1$, since $\Omega(n)\sim n|\check \lambda|$. Therefore, to prove the induction step, it suffices to prove that the expression is less than one for $n=n_1$. Thus, the close to identity change of variables converges to an analytic function on $|x|<r_\phi=s$.
\end{proof}

\subsection{Computing $G$ and its radius of convergence}
It was proved above that there exists an analytic change of variables $\phi$ on $|x|<r_\phi$. After changing the coordinates, the system is of the form 
$$\dot x = \Lambda x + G(x),$$ where $[G]_{\mathbb V_l}=0$. We want to estimate the radius of convergence of $G$.

Let $\hat \phi$ be as in the above section, then a majorant for $G$ is given by $\hat G$, defined by 
\begin{equation}\label{GR}
\hat g_k=[\hat F(x+\hat \phi^{k-1})+2(\hat \phi^{k+1-2l})'\hat G^{k-1}]_k,
\end{equation}
where we note that $\hat g_k=0$, for $0\leq k < 2l$, since $[\hat G]_{\mathbb V_l}=0$. The $\hat g_k$'s are computed using automatic differentiation.

We use (\ref{GR}) to compute $\hat g_k$, for $2l\leq k\leq N_G$, where $N_G$ is an integer larger than $2l$. These values are used to compute candidate constants $D$ and $K$, such that $\hat g_k\leq DK^k$. To do this, we again use our least squares estimator.
We require that $K>M$; the reason being that $\frac{1}{K}$ will be used to estimate the radius of convergence for $G$, and $G$ is only of interest within the radius of convergence of $\phi$.
Let $$\begin{array}{cccc}
\Psi(n) & := & & A\left(\left(2\sum_{k=1}^{n_0}\hat \alpha_k M^{-k}+C(n-2n_0)\right)\frac{C}{D}\left(\frac{M}{K}\right)^{n+1}\right) \\
& & + & 2\left(\sum_{k=2}^{n_0} k\hat \alpha_k K^{1-k} + CM \left(\frac{M}{K}\right)^{n_0}\frac{(n_0+1)-(n_0)\frac{M}{K}}{\left(1-\frac{M}{K}\right)^2}\right)
\end{array}$$

\begin{proposition}\label{GConvProp}
If $\Psi(N_G) < 1,$ then $G$ is analytic on $\mathfrak{B}_{{K}^{-1}}$.
\end{proposition}
\begin{proof}
Assume that we have proved $\hat g_k\leq DK^k$, for $k\leq n$. If we can prove that $\hat g_{n+1}\leq DK^{n+1} $, the convergence of $G$ follows by induction. As in the proof of Proposition \ref{phiConvProp}, we use the constant $A$ to get a bound on $\hat F$, which gives us the bound
$$\hat g_{n+1} \leq A\sum_{k=1}^n \hat \alpha_k \hat \alpha_{n+1-k}+2\sum_{k=2}^{n+2-2l}k\hat \alpha_k \hat g_{n+2-k}.$$
We call the first sum $\Sigma_1$, and the second sum $\Sigma_2$. If we can prove that $A\Sigma_1+2\Sigma_2$ is bounded by $\Psi(n) DK^{n+1}$, where $\Psi: \mathbb{N}\to \mathbb{R}$ is a decreasing function, we are done.

$$\begin{array}{lll}
\Sigma_1 &\leq& \sum_{k=1}^{n_0}\hat \alpha_k CM^{n+1-k}+\sum_{n_0+1}^{n-n_0}C^2M^{n+1}+\sum_{n-n_0+1}^n CM^k \hat \alpha_{n+1-k} \\
 &=& \left(2\sum_{k=1}^{n_0}\hat \alpha_k M^{-k}+C(n-2n_0)\right)CM^{n+1} \\
 &\leq&\left(\left(2\sum_{k=1}^{n_0}\hat \alpha_k M^{-k}+C(n-2n_0)\right)\frac{C}{D}\left(\frac{M}{K}\right)^{n+1}\right)DK^{n+1},
\end{array}
$$
since $K>M$, the bound on $\Sigma_1$ is decreasing in $n$.

$$\begin{array}{lll}
\Sigma_2 & \leq & D\left(\sum_{k=2}^{n_0} k\hat \alpha_k K^{n+2-k} + C\sum_{k=n_0+1}^{n+2-2l}kM^kK^{n+2-k}\right) \\
 & \leq & \left(\sum_{k=2}^{n_0} k\hat \alpha_k K^{1-k} + CM\sum_{k=n_0+1}^{\infty}k\left(\frac{M}{K}\right)^{k-1}\right)DK^{n+1} \\
& = & \left(\sum_{k=2}^{n_0} k\hat \alpha_k K^{1-k} + CM \left(\frac{M}{K}\right)^{n_0}\frac{(n_0+1)-(n_0)\frac{M}{K}}{\left(1-\frac{M}{K}\right)^2}\right)DK^{n+1}
\end{array}
$$
the expression in front of $DK^{n+1}$ is independent of $n$. Thus, $\Psi$ is a decreasing function, which proves the bound $\hat g_k\leq DK^{k}$ for all $k$. The analyticity of $G$ on $r_1:=\frac{1}{K}$ follows.
\end{proof}

\subsection{Computing the bounds}
Let, $r_0=\epsilon_\phi r_\phi,$ $r_2=\epsilon_G r_1,$ and $r_3=\epsilon \min(r_0,r_2),$ 
where $0<\epsilon_\phi, \epsilon_G,\epsilon<1$, are given numbers.

That is, $\mathfrak{B}_{r_0}$ is the domain of $\phi$ that we will use to estimate $||\phi||$ and $||\phi^{-1}||$, and $\mathfrak{B}_{r_2}$ is the domain of $G$ that we will use to estimate $\kappa$. To ensure that our estimates hold we may never leave these domains. 
$\mathfrak{B}_{r_3}$ is the box where we will actually change coordinates.

To guarantee that the change of variables is done in the domain of $G$, we need that $r_3+r_3^2 K_0 < r_2$. By construction, the flow stays inside the domain of $G$, since the only place that the flow can leave the box is on the unstable side. To guarantee that the final change of coordinates is done in the part of the domain of $\phi$, where the estimate on $||\phi^{-1}||$ holds, we need that $r_3<r_0(1-K_0r_0)$, since then $r^*<r_0$, where $r^*$ is such that $r_3=r^*-||\phi||_{r^*}$. 

To compute $K_0$ we note that on $|x|<r_0$, we have that 
$$\begin{array}{lllll}
||\phi||_r & \leq & \hat \phi(r) & = & \sum_{k=2}^{n_0} \hat \alpha_k r^k+C\sum_{k=n_0+1}^\infty M^k r^k \\
& & & \leq & r^2\left(\sum_{k=2}^{n_0}\hat \alpha_kr_0^{k-2}+CM^2\sum_{k=n_0-1}^\infty M^kr_0^k\right) \\
& & & \leq & r^2\left(\sum_{k=2}^{n_0}\hat \alpha_kr_0^{k-2}+CM^2\frac{(Mr_0)^{n_0-1}}{1-Mr_0}\right).
\end{array}
$$
Thus, if we put
\begin{equation}\label{K0}
K_0:=\sum_{k=2}^{n_0}\hat \alpha_kr_0^{k-2}+CM^2\frac{(Mr_0)^{n_0-1}}{1-Mr_0},
\end{equation}
then  $$||\phi||_{r}\leq K_0r^2.$$

To estimate $||\phi^{-1}||_r$, we need to find $r^*$ such that $r=r^*-||\phi||_{r^*}$ since then, by Lemma \ref{phiInvN}, 
$||\phi^{-1}||_{r}\leq||\phi||_{r^*}$. 
A trivial calculation yield 
$$r^*=\frac{1}{2K_0}-r-\sqrt{\frac{1}{4K_0^2}-\frac{r}{K_0}}.$$
Thus, 
\begin{equation}\label{phiInv}
||\phi^{-1}||_{r}\leq K_0 (r^*)^2 \leq \frac{1}{2K_0}-r-\sqrt{\frac{1}{4K_0^2}-\frac{r}{K_0}}.
\end{equation}

Finally, the constant $\kappa$ is computed as
$$\kappa:=\frac{DK^{2l}}{1-Kr_2}r_2^{2l-1}.$$
We want that $\kappa\ll\min(-\lambda_s,\lambda_u,|\lambda_s+\lambda_u|)$; if this is not the case, we decrease $r_2$ and/or $r_3$.

\section{Examples}
\subsection{Example 1}
We start with a simple example that also illustrates how the results depend on the distance from resonance. The vector field under study is 
\begin{equation}\label{Ex1Eq}
\begin{array}{ccl} 
\dot x_s & = & -x_u \\
\dot x_u & = & x_s^3+0.05x_s^2-0.95x_s\\ & + & \delta((438.4905-25.2469x_s-452.7899x_s^2)x_u-741.0341x_u^3/3)
\end{array}
\end{equation}
which has previously been examined in \cite{JT08}. It is a perturbation of a Hamiltonian system, given by $\delta=0$. The Hamiltonian system has a resonance of flat-order $1$, since at a saddle of a planar Hamiltonian vector field the stable and unstable eigenvalues have the same modulus. 

We describe the results in detail for $\delta=10^{-3}$, and also include Table \ref{etaDep}, which illustrates how the convergence radii, and norm bounds depend on the distance from the resonance. We have chosen $l=10$, since this is the lowest value of $l$ that, after optimisation of $\epsilon_\phi = 0.1, \epsilon_G = 0.5,$ and $\epsilon =  0.9$, yields $\kappa<2^{-53}$, which is the machine precision using IEEE double precision floating point arithmetic.

We start by introducing the linear change of variables, $x=T\xi$, that transforms (\ref{Ex1Eq}) to the form (\ref{DEq}). Note that this transformation yields interval enclosures of the eigenvalues 
$$\lambda_s=-0.77978852302649^{94}_{89}, \quad \lambda_u=1.21827902302649_{88}^{95},$$ 
which are used during the computations. The diagonalised system is put into the algorithm. 

The algorithm starts by verifying that $\Lambda \in \mathcal F_{10},$ and computes $N(10) = 25$. Using $\mu=0.2$, and $\eta=0.08$, yields $n_0 = 13$, and $n_1 = 30$. Therefore, we need to internally represent all functions by their Taylor models of order $30$. 
Next, $\phi_u^{30}$, and $\phi_s^{30}$ are computed using the recursive formula (\ref{compAlpha}), and used to compute $\hat \phi^{30}$. The least squares estimator of the coefficients of $\hat \phi^{30}$ yields 
$$\hat \alpha_k \leq 0.08\times 397^k, \quad n_0<k\leq n_1,$$
i.e. $C = 0.08$, and $M = 397$. We compute $\hat F^{30}$, and use $s=\frac{1}{M}$ as the candidate radius of convergence to compute $A=1.64$. To prove that $\phi$ converges we verify Proposition \ref{phiConvProp}. This yields $r_\phi = 2.52\times 10^{-3}$, $r_0=2.52 \times 10^{-4}$, and Equation \ref{K0} gives $K_0 = 22.6$.

The algorithm now turns to the majorisation of $G$. We compute $\hat g_k$, for $2l \leq k \leq 4l$, using the recursive formula (\ref{GR}). The least squares estimator of the coefficients of $\hat G$, yields 
$$\hat g_k \leq 1.03 \times10^{-18} \times 2490^k, \quad 2l\leq k\leq 4l,$$
i.e. $D = 1.03 \times10^{-18}$, and $K=2490^k$. These values are used to verify Proposition \ref{GConvProp}. This yields $r_1 = 4.02\times 10^{-4}$, 
$r_2 = 2.01\times10^{-4}$, $r_3 = 1.81\times 10^{-4}$, and $\kappa = 9.74\times 10^{-21}$.

Finally, we verify that we compute within the domains of validity of the constants $K_0$, and $\kappa$. When we enter the box $|\xi|<r_3$, we apply $\phi$, which alters the coefficient by at the most $K_0 r_3^2$, that is we might start computing with $y_u = (1+K_0 r_3)r_3 < 1.82\times 10^{-4}$. Inequality \ref{thm2} gives the bound, $y_s=\psi(y, \tau_e(y) \leq (1+K_0 r_3)r_3 
\left(\frac{(1+K_0 r_3)r_3}{r_3}\right)^{\frac{|\lambda_s|+\kappa}{\lambda_u-\kappa}} \leq 1.84\times 10^{-4}$. We use Equation \ref{phiInv}, 
and compute $||\phi^{-1}||_{y_s} \leq \frac{1}{2K_0}-y_s-\sqrt{\frac{1}{4K_0^2}-\frac{y_s}{K_0}} \leq 8\times 10^{-7}$. 
Thus, the flow exits the computations inside of the box $|\xi|<1.85\times 10^{-4}$, which is inside of $|\xi|<\min(r_0, r_2, r_0(1-K_0r_0))=r_2=2.01\times10^{-4}$, where the bounds on $K_0$, $\kappa$, and $||\phi^{-1}||$ are valid.

\begin{table}[h]
\begin{center}
\begin{tabular}{c|ccccc}
$\delta$ & $r_\phi$ & $r_1$  & $r_3$ & $K_0$ & $\kappa$ \\ \hline \\
$10^{-3}$ & $2.52\times 10^{-3 }$ & $ 4.02\times 10^{-4 }$ & $1.81\times 10^{-4 }$ & $ 22.6$ & $ 9.74\times 10^{-21}$ \\
$10^{-5}$ & $6.85\times 10^{-2 }$ & $1.35\times 10^{-2 }$ & $6.09\times 10^{-3 }$ & $ 6.52$ & $ 7.17\times 10^{-19}$ \\
$10^{-7}$ & $9.97\times 10^{-3}$ & $ 1.46\times 10^{-3 }$ & $ 6.59\times 10^{-4}$ & $ 7.45\times 10^{+1 }$ & $ 3.15\times 10^{-20}$ \\
$10^{-9}$ & $1.14\times 10^{-3}$ & $ 1.46\times 10^{-4 }$ & $ 6.59\times 10^{-5}$ & $ 8.42\times 10^{+2 }$ & $ 3.13\times 10^{-21}$ \\
$10^{-11}$ & $1.30\times 10^{-4 }$ & $1.46\times 10^{-5 }$ & $6.59\times 10^{-6 }$ & $ 9.67\times 10^{+3}$ & $ 3.14\times 10^{-22}$ \\
$10^{-13}$ & $ 1.21\times 10^{-5}$ & $1.45\times 10^{-6 }$ & $6.53\times 10^{-7 }$ & $8.94\times 10^{+4 }$ & $4.46\times 10^{-23}$  \\
$10^{-15}$ & $1.07\times 10^{-6 }$ & $ 1.28\times 10^{-7}$ & $ 5.79\times 10^{-8}$ & $ 7.90\times 10^{+5}$ & $ 1.37\times 10^{-22}$ \\
$10^{-17}$ & $ 8.62\times 10^{-8}$ & $1.11\times 10^{-8 }$ & $5.02\times 10^{-9 }$ & $ 8.28\times 10^{+6}$ & $ 3.08\times 10^{-20}$ \\
\end{tabular}
\vspace{0.1in}
\caption{Convergence radii and norm estimates in Example 1 as $\delta$ is varied.}\label{etaDep}
\end{center}
\end{table}

\subsection{Example 2}
In our second example, we follow a solution curve close to a graphic. A graphic is an invariant set of a flow consisting of saddles and separatrices, see e.g. \cite{R98}. 
Consider the following vector field
\begin{equation}\label{ex2}
\begin{array}{ccl}
\dot x & = & (\delta x+y)(x^2 - 1) \\
\dot y & = & (-x+\delta y)(y^2 - 1)
\end{array},
\end{equation}
where we will consider $\delta= -0.2$. If $\delta = 0$, this is a Hamiltonian field with the first integral $H=\frac{-y^2x^2+x^2+y^2}{2}$. There are five critical points, an unstable focus [centre if $\delta=0$] at the origin and four saddles at $(\pm 1,\pm 1)$, see Figure \ref{graphicPic}. This example is simple enough so that we can determine most qualitative properties by hand, which allows us to focus our attention to the application of our algorithm.

\begin{figure}[]
\begin{center}
\includegraphics[width=0.5\textwidth]{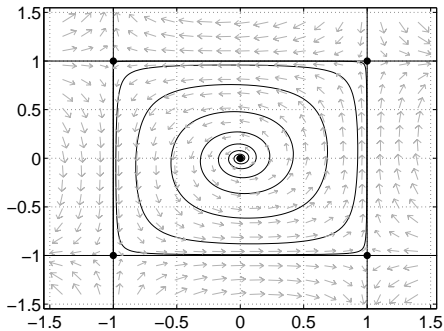}
\caption{Phase portrait of the system from Example 2.}\label{graphicPic}
\end{center}
 \end{figure}

The curves $x=\pm 1$ and $y=\pm 1$ are invariant under the flow of (\ref{ex2}). In fact, they are the separatrices of the saddles. We only consider the flow inside of the graphic. For $\delta = 0$, $H=0$ corresponds to the origin, and $H=\frac{1}{2}$ to the graphic. To determine the properties of the flow of (\ref{ex2}), we begin by noting that, for $\delta<0$, the vector field is transversal to the solution curves of the unperturbed system. Indeed, let $r^2 = \frac{-y^2x^2+x^2+y^2}{2}$, then 
$$ \begin{array}{ccl}
2r\dot r & = & -y^2x\dot x-yx^2\dot y+y\dot y+x\dot x \\  
& = & (-y^2x+x)(\delta x+y)(x^2-1)+(-yx^2+y)(-x+\delta y)(y^2-1) \\
& = & \delta(x(x^2-1)(-y^2x+x)+y(y^2-1)(-yx^2+y)) \\
& = & \delta(x^2(-y^2x^2+x^2)+y^2(-y^2x^2+Y^2)-2r^2+x^2y^2) \\
& = & \delta(2r^2(x^2+y^2)-2r^2-x^2y^2)=\delta(x^2+y^2)(2r^2-1) \\
& > & 0
\end{array}
$$

Thus, if we leave the neighbourhood of one saddle on $H=C$, we enter the neighbourhood of the next one outside of $H=C$. It follows that we do not need any numerical integrator to estimate the distance to the graphic from above. We will start in a neighbourhood of a saddle, use our computed analytical estimates to pass it and then enter at the next one at the same level curve of $H$. In addition we only need to consider one of the saddles, since the system is symmetric. We therefore translate $(-1,-1)$ to the origin and get the system:
\begin{equation}
\begin{array}{ccl}
\dot x_s & = & -2.2x_s-x_ux_s^2-0.1x_s^3+2x_ux_s+1.3x_s^2 \\
\dot x_u & = & 1.8x_u-0.1x_u^3+x_u^2x_s-0.7x_u^2-2x_ux_s
\end{array},
\end{equation}

Our program yields the output shown in Table \ref{ex2Out}.
\begin{table}[h]
\begin{small}
\begin{verbatim}
Resonance of order 9 detected
l = 9, N_l = 19
Order of Taylor approximations = 23
n_0 = 10, n_1 = 23
C = [0.0180,0.0181]  M= [3.9853,3.9854]
A = [4.5520,4.5521]
Phi is analytic on the disk with radius = [0.2509, 0.2510]
K_0 <= [2.3223,2.3224] on the disk r_0 = [0.0376,0.0377]
D = [1.2701E-010,1.2702E-010] K = [12.5616,12.5617]
G is analytic on the disk with radius = [0.0796,0.0797]
kappa <= [1.5544E-017,1.5545E-017] on the disk r_2 = [0.0262,0.0263]
We recommend that you change to the normal form
on the disk with radius r_3 = [0.0210,0.0211]
\end{verbatim}
\end{small}
\caption{The output generated by the program in example 2.}\label{ex2Out}
\end{table}

We will change to the normal form on $\mathfrak B_{0.02}$, i.e. $r=0.02$, and consider the trajectory that starts at $(x_s, x_u) = (0.02, 0.01)$ in the translated coordinate system. By Theorem \ref{thm2}, together with the bounds on $K_0$ and $\kappa$, we can calculate where it will leave $\mathfrak B_{0.02}$. We start the calculation with $y_i<(1+K_0 x_i)x_i$, and then use our bounds on the flow inside $\mathfrak{B}_{0.02}$, to get the following bound at $y_u=0.02$
$$y_s\leq (1+K_0 r)r \left(\frac{(1+K_0 x_u)x_u}{r}\right)^{\frac{|\lambda_s|-\kappa}{\lambda_u+\kappa}}.$$ 
Thus, on the outgoing stable coordinate we have the following bound,
$$x_s\leq (1+K_0 r)r \left(\frac{(1+K_0 x_u)x_u}{r}\right)^{\frac{|\lambda_s|-\kappa}{\lambda_u+\kappa}}+||\phi^{-1}||.
$$
By the transversality and symmetry properties of the system, we enter the neighbourhood of the next saddle outside of $(r,x_s)$, and so on. If we follow our trajectory we get the upper bounds on its distance from the graphic, and lower bounds on the lap times, shown in Table \ref{graphConv}.

\begin{table}[h]
\begin{center}
\begin{tabular}{c|lc}
$lap$ & $x_u$ & $laptime$ \\ \hline
0 & 0.01 & 0 \\
1 & $7.1 \times 10^{-3}$ & 1.7\\
2 & $3.0 \times 10^{-3}$ & 2.8\\
3 & $3.8 \times 10^{-4}$ & 5.6\\
4 & $3.7\times 10^{-6}$ & 12\\
5 & $1.2\times10^{-10}$& 26\\
6 & $3.0\times 10^{-17} $& 58\\

\end{tabular}
\caption{Converging to the graphic.}\label{graphConv}
\end{center}
\end{table}

To compare with the performance of a standard numerical integrator we do the same computations using the {\tt ode45} solver in MATLAB, which incorrectly starts fluctuating around $x_u=10^{-7}$, the result is shown in Table \ref{graphConvNum}.

\begin{table}[h]
\begin{center}
\begin{tabular}{c|lc}
$lap$ & $x_u$ & $laptime$ \\ \hline
0 & 0.01 & 0 \\
1 & $3.9 \times 10^{-7} $& 23 \\
2 & $8.0 \times 10^{-8} $& 38 \\
3 & $1.3 \times 10^{-7} $& 39 \\
4 & $7.9 \times 10^{-8} $& 39 \\
5 & $9.8 \times 10^{-8} $& 39 \\
6 & $1.3 \times 10^{-7} $& 39 \\
\end{tabular}
\caption{Numerical integration close to the graphic.}\label{graphConvNum}
\end{center}
\end{table}
\begin{small}

\end{small}

\begin{thebibliography}{10}
\bibitem{AH83} G.\,Alefeld, and J.\, Herzberger, 
Introduction to Interval Computations, 
Academic Press, New York, 1983.

\bibitem{BHK05} F.\,Blomquist, W.\,Hofschuster, W.\, Kr\"amer, Real and Complex Taylor Arithmetic in C-XSC
Preprint 2005/4, Universit\"at Wuppertal, 2005
Available from {http://www.math.uni-wuppertal.de/~xsc}

\bibitem{CXSC}
CXSC -- C++ eXtension for Scientific Computation, version 2.0.
Available from {http://www.math.uni-wuppertal.de/~xsc}

\bibitem{G00} A.Griewank,
Evaluating derivatives: Principles and techniques of algorithmic differentiation,
SIAM Frontiers in Applied Mathematics, 19, Philadelphia, 2000.

\bibitem{HH95} R.\,Hammer, M.\,Hocks, U.\,Kulisch, and D.\,Ratz,  
C++ Toolbox for Verified Computing, 
Springer-Verlag, New York, 1995.

\bibitem{H76} E.\,Hille, Ordinary differential equations in the complex domain. Pure and Applied Mathematics. Wiley-Interscience [John Wiley \& Sons], New York-London-Sydney, 1976. xi+484 pp

\bibitem{JT08} T.\,Johnson, W.\,Tucker, On a computer-aided approach to the computation of Abelian integrals, submitted.

\bibitem{Mo66} R.E.\,Moore, Interval Analysis, 
Prentice-Hall, Englewood Cliffs, New Jersey, 1966.

\bibitem{Mo79} R.E.\,Moore, 
Methods and Applications of Interval Analysis, 
SIAM Studies in Applied Mathematics, Philadelphia, 1979.

\bibitem{Ne90} A.\,Neumaier,
Interval Methods for Systems of Equations.
Encyclopedia of Mathematics and its Applications 37,
Cambridge Univ. Press, Cambridge, 1990

\bibitem{R98} R.\,Roussarie, Bifurcation of planar vector fields and Hilbert's sixteenth problem. Progress in Mathematics, 164. Birkh\"auser Verlag, Basel, 1998.

\bibitem{S52} C.L.\,Siegel, \"Uber die Normalform analytischer Differentialgleichungen in der N\"ahe einer Gleichgewichtsl\"osung. Nachr. Akad. Wiss. G\"ottingen. Math.-Phys. Kl. Math.-Phys.-Chem. Abt.  (1952) 21--30.

\bibitem{SM71} C.L.\,Siegel, J.K.\,Moser, Lectures on celestial mechanics. Translation by Charles I. Kalme. Die Grundlehren der mathematischen Wissenschaften, Band 187. Springer-Verlag, New York-Heidelberg, 1971. xii+290 pp

\bibitem{T02} W.\,Tucker, A rigorous ODE solver and Smale's 14th problem.  Found. Comput. Math.  2  (2002),  no. 1, 53--117.

\bibitem{T04} W.\, Tucker, Robust normal forms for saddles of analytic vector fields.
Nonlinearity, 17, pp. 1965-1983, 2004.

\end{thebibliography}
\end{document}